\documentclass{amsart}
\usepackage{amssymb, latexsym,array}
\theoremstyle{plain}
\newtheorem{theorem}{Theorem}
\newtheorem{corollary}{Corollary}
\newtheorem{lemma}{Lemma}

\theoremstyle{definition}
\newtheorem{definition}{Definition}
\begin{document}
\title[Algebraically defined Costas arrays]
{Parity properties of Costas arrays\\
defined via finite fields}

\author[Konstantinos Drakakis]{Konstantinos Drakakis${}^1$}
\thanks{${}^1$ School of Electrical, Electronic and Mechanical Engineering,
University College Dublin,
Belfield, Dublin 4,
Ireland,
\emph{E-mail address:} \texttt{konstantinos.drakakis@ucd.ie}}

\author[Rod Gow]{Rod Gow${}^2$}
\thanks{${}^2$ School of Mathematical Sciences,
University College Dublin,
Belfield, Dublin 4,
Ireland,
\emph{E-mail address:} \texttt{rod.gow@ucd.ie}}

\author[Scott Rickard]{Scott Rickard${}^3$}
\thanks{${}^3$ School of Electrical, Electronic and Mechanical Engineering,
University College Dublin,
Belfield, Dublin 4,
Ireland,
\emph{E-mail address:} \texttt{scott.rickard@ucd.ie}}

\thanks{All authors are also affiliated with the Claude Shannon Institute (www.shannoninstitute.ie).}

\thanks{This work is a unified and expanded version of \cite{D2,G}.}

\begin{abstract} A Costas array of order $n$ is an arrangement of dots and blanks into $n$ rows
and $n$ columns, with exactly one dot in each row and each column, the arrangement satisfying certain specified conditions.
A dot occurring in such an array is even/even if it occurs in the $i$-th row and $j$-th column, where $i$ and $j$ are both even
integers, and there are similar definitions of odd/odd, even/odd and odd/even dots. Two types of
Costas arrays, known as Golomb-Costas and Welch-Costas arrays, can be defined using finite fields.
When $q$ is a power of an odd prime, we enumerate the number of even/even
odd/odd, even/odd and odd/even dots in a  Golomb-Costas array.
We show that three of these numbers are equal and they differ by $\pm 1$ from the fourth.
For a Welch-Costas array of order $p-1$, where $p$ is an odd prime, the four numbers above
are all equal to $(p-1)/4$ when $p\equiv 1\pmod{4}$, but when $p\equiv 3\pmod{4}$, we show that
the four numbers are defined in terms of the class number of the imaginary
quadratic field $\mathbb{Q}(\sqrt{-p})$, and thus behave in a much less predictable manner.
\end{abstract}
\maketitle
\section{Introduction}
\noindent

\noindent A \textit{Costas array}  of order $n$
is an $n\times n$ array of dots and blanks satisfying the two conditions:
\smallskip
\begin{itemize}

\item  there are $n$ dots in the array, one in each row and in each column;
\item no two of the ${n\choose 2}$ line segments joining two dots in the array have the same
length and slope.
\end{itemize}
\smallskip

We shall call an $n\times n$ array of dots and blanks satisfying the first condition a
\textit{ permutation array}, since it corresponds to the well known concept of a permutation matrix.
A Costas array of order $n$ determines a special type of
permutation, $\pi$, say,  of the numbers $\{\,1,2, \ldots, n\,\}$. Specifically, if we have a dot in the $(i,j)$ position of the array, we set
 \[
 \pi(i)=j.
\]
 The second of the Costas array conditions then implies that
 \[
 \pi(i+k)-\pi(i)\ne \pi(j+k)-\pi(j)
 \]
 whenever $i\ne j$ and $1\le i<i+k\le n$,  $1\le j<j+k\le n$.

Costas arrays were introduced by the electrical engineer John Costas in  the 1960's. They create
ideal waveforms for certain sensor applications, reducing ambiguity in interpreting radar and sonar
returns. See, for example, \cite{Gol}, pp.188-190, and \cite{D}.
For an arbitrary positive integer $n$, we do not know any systematic ways of
constructing Costas arrays of order $n$, and it is an open question whether Costas arrays
of order $n$ actually exist for all
integers $n\ge 1$. Furthermore, in those cases where Costas arrays of order $n$ are known to exist,  they seem to form a very small fraction of the total number of permutation arrays of the same order. Thus, while there are clearly $26!$ permutation arrays of order 26, exhaustive
computer searching has shown that only 56 of them are Costas arrays. We have
a very explicit condition for  a permutation to define a Costas array, but in practice, it is a condition which is difficult to exploit in terms of traditional algebraic properties of permutations, and no simple criteria
have emerged which enable us to distinguish Costas permutations from ordinary permutations.

In the search for Costas arrays, it is reasonable to ask whether, from a given Costas array of order $n$, we can construct a subarray (or superarray) which has the Costas property. For example if the Costas array  has a dot at the $(1,1)$ position (a so-called corner dot), we can then remove the first row and column
to obtain a Costas array of order $n-1$. This procedure, and similar variants, has frequently been employed to obtain Costas arrays of smaller size. Somewhat harder to achieve, is to add a corner dot
to create a Costas array of order $n+1$. An example of this procedure is the creation of a Costas array
of order 31 from one of order 30.

The motivation of the present paper, which relates to constructing
Costas arrays from larger such arrays, is described by the following simple lemma.

\begin{lemma} \label{motlem} Let $C$ be a Costas array of order $n$.
Let $\Delta$ be the subset of ordered pairs $(i,j)$, where there is a dot in the $(i,j)$-position
of $C$. Suppose that there is an integer $r>1$ such that
whenever $(i,j)$ is in $\Delta$ and $r$ divides $i$, then $r$ also divides $j$. Write $n=rm+d$, where
$m$ and $d$ are integers and $0\le d<r$. Then we may construct a Costas array of order $m$ by placing
a dot in the $(u,v)$ position, where $v$ is defined as the unique positive integer such that
$(ru,rv)\in\Delta$.
\end{lemma}

\begin{proof} We let $\Omega $ denote the set of integers lying between 1 and $n$, and
$\Omega'$ denote the subset of integers lying between 1 and $m$.
Let $\sigma$ be the permutation of $\Omega$ corresponding to $C$.
We define
a permutation $\sigma'$ of $\Omega'$ by setting
\[
\sigma'(u)=\frac{\sigma(ru)}{r}
\]
We claim that $\sigma'$ has the Costas property. For suppose that
 \[
 \sigma'(u+w)-\sigma'(u)= \sigma'(v+w)-\sigma'(v)
 \]
 where $1\le u<u+w\le m$,  $1\le v<v+w\le m$. It follows then by definition that
 \[
 \sigma(ru+rw)-\sigma(ru)= \sigma(rv+rw)-\sigma(rv)
 \]
 But $r\le ru<ru+rw\le rm\le n$,  $r\le rv<rv+rw\le rm\le n$, and since $\sigma$ has the Costas property,
 it follows that $ru=rv$. Hence $u=v$, which implies that $\sigma'$ also has the Costas property.
 \end{proof}

In the simplest case of Lemma \ref{motlem}, namely when $r=2$, we could (in principle) construct two smaller Costas arrays, one by retaining the dots whose coordinates are both even, and the other by retaining the dots whose coordinates are both odd. However, such constructions are impossible, as the following theorem indicates \cite{DRG,FL}:

\begin{theorem} \label{lev}
In a Costas array of order $n>2$, at least one dot has one even and one odd coordinate.
\end{theorem}

We will provide examples  of larger values of $r$ for which the hypothesis of Lemma \ref{motlem} is fulfilled, although we must admit that the corresponding subarray derived is already known.

%



 Having now discussed our general intentions for investigating Costas arrays, we proceed to
 describe and investigate the field-theoretic constructions of a special class of these arrays.

 \section{Properties of Costas arrays constructed by field-theoretic methods }
\noindent
For certain special values of $n$, finite fields provide a framework to produce Costas arrays of order $n$,
and we intend to show in this paper how the arrays so produced relate to the phenomena described in
Lemmas 1 and 2.  We begin by  explaining the field-theoretic constructions, as presented in \cite{GolTay}.

Let $p$ be a prime integer and let $q=p^f$, where $f$ is a positive integer. Let
$\mathbb{F}_q$ denote the finite field of size $q$ and let $\mathbb{F}_q^*$ denote the multiplicative
group of non-zero elements in $\mathbb{F}_q$.  Since $\mathbb{F}_q^*$ is a cyclic
group, we may take any two generators $\alpha$ and $\beta$  of it (we allow the
possibility that $\alpha=\beta$) and construct a corresponding permutation array
of order $q-2$ by putting a dot in the position $(i,j)$, where $1\le i, j\le q-2$, whenever
\[
\alpha^i+\beta^j=1.
\]
It is well known that such an array has the Costas property. We shall call an array constructed by this
procedure a \textit{Golomb-Costas array} of order $q-2$.

Suppose that the integer $f$ can be factorized as a product $f=de$, where $d$ and $e$ are positive integers. Write $Q=p^d$. Then $q=Q^e$
and $\mathbb{F}_Q$ is a subfield
of index $e$ in $\mathbb{F}_q$. We set
\[
r=\frac{Q^e-1}{Q-1}.
\]
Let $\alpha$ be a generator of $\mathbb{F}_q^*$. It is straightforward
to prove that $\alpha^i\in \mathbb{F}_Q$ if and only if $r$ divides $i$.

\begin{lemma} \label{divprop} Let $q=Q^e$, where $Q=p^d$ is a power of a prime $p$, and let $\alpha$ and $\beta$ be generators of $\mathbb{F}_q^*$.
Let
\[
r=\frac{Q^e-1}{Q-1}.
\]
Suppose that for some integers $i$ and $j$ we have
\[
\alpha^i+\beta^j=1.
\]
Then $r$ divides $i$ if and only if $r$ divides $j$.
\end{lemma}

\begin{proof}
This follows from our observation above describing when $\alpha^i$ is in $\mathbb{F}_Q$ and
also from the obvious
fact that, under our given hypothesis on $\alpha$ and $\beta$, $\alpha^i\in \mathbb{F}_Q$ if and only if
$\beta^j\in \mathbb{F}_Q$.
\end{proof}

Lemma  \ref{divprop} implies that if $q=p^f$, where $f$ is an integer greater than 1,  a Golomb-Costas
array of order $q-2$ always satifies the hypotheses of Lemma \ref{motlem}. The Costas array constructed
by the lemma is simply a Golomb-Costas array of order $p^e-2$, where $e$ is any positive integer
divisor of $f$. Thus the subarray is not new.  (We note that the case when $q=4$ leads to a vacuous
substructure.)

For each odd prime $p$, there is also
a construction, due to  L. R. Welch, of a Costas array of order $p-1$. We  define it as a permutation
on the set $\Omega=\{\,1,2, \ldots, p-1\,\}$ as follows. Let $\alpha$, $c$ be any
generator,  any element of  $\mathbb{F}_p^*$, respectively, and define
\[
\sigma=
\sigma_{\alpha, c}: \Omega\to\Omega
\]
by $\sigma(i)=j$, where $j$ is that element of $\Omega$ satisfying
\[
c\alpha^i\equiv j\pmod {p}.
\]
It is straightforward to see that all these $\sigma_{\alpha, c}$ are different, and hence there are
\[
(p-1)\phi(p-1)
\]
such permutations. Furthermore, each permutation $\sigma_{\alpha, c}$ has the Costas property
and thus we obtain $(p-1)\phi(p-1)$  corresponding Costas arrays. We will call these
\textit{Welch-Costas arrays}. It is not at all easy to see whether any Welch-Costas array has the property
described in Lemma \ref{motlem}, for a particular value of $r$.

\section{Parity properties of the dots of a Costas array}

\noindent We introduce here a simple concept to generalize the idea implicit in Theorem \ref{lev}.
We say that a dot occurring in a permutation array is
even/even if it occurs in the $i$-th row and $j$-th column, where $i$ and $j$ are both even
integers. We may likewise define odd/odd, even/odd and odd/even dots. In this connection, the following notation is convenient to use.

\begin{definition}  Let $P$ be a permutation array of order $n$. We let
\[
\#(e,e),\quad \#(o,o),\quad \#(e,o),\quad \#(o,e)
\]
denote the number of
even/even, odd/odd, even/odd and odd/even dots in $P$ respectively
\end{definition}

Note that we shall show in Lemma \ref{parpop} that $\#(e,o)=\#(o,e)$. In the light of the notation just introduced, all that Theorem \ref{lev} states is that, for any Costas array of order greater than 2, $ \#(o,e)>0$.

All Costas arrays of order $n$, where $n\le 26$, are known; there are 143, 635 of them. For the convenience of the reader, we tabulate statistical information about the location of the dots in Costas
arrays in terms of the numbers $\#(e,e)$, etc. Table \ref{stat} displays  this information for order $n\le 15$, 15 being the highest order we know for which Costas arrays with $\#(o,e)=1$ exist. We are inclined to speculate that, as $n\rightarrow \infty$, $\#(e,o)\rightarrow \infty$ for Costas arrays of order $n$.

\begin{table}
\[
\begin{array}{|r|rrrr|r|}\hline
n & \#(e,e) & \#(o,o) & \#(e,o) & \#(o,e) & \text{\# of arrays}\\ \hline
2 & 0 & 0 & 1 & 1 & 1\\
   & 1 & 1 & 0 & 0 & 1\\ \hline
3 & 0 & 1 & 1 & 1 & 4\\ \hline
4 & 1 & 1 & 1 & 1 & 12\\ \hline
5 & 0 & 1 & 2 & 2 & 16\\
   & 1 & 2 & 1 & 1 & 24\\ \hline
6 & 1 & 1 & 2 & 2 & 58\\
   & 2 & 2 & 1 & 1 & 58\\ \hline
7 & 0 & 1 & 3 & 3 & 20\\
   & 1 & 2 & 2 & 2 & 68\\
   & 2 & 3 & 1 & 1 & 112\\ \hline
8 & 1 & 1 & 3 & 3 & 138\\
   & 2 & 2 & 2 & 2 & 168\\
   & 3 & 3 & 1 & 1 & 138\\ \hline
9 & 0 & 1 & 4 & 4 & 16\\
   & 1 & 2 & 3 & 3 & 320\\
   & 2 & 3 & 2 & 2 & 212\\
   & 3 & 4 & 1 & 1 & 212\\ \hline
10 & 1 & 1 & 4 & 4 & 304\\
      & 2 & 2 & 3 & 3 & 776\\
      & 3 & 3 & 2 & 2 & 776\\
      & 4 & 4 & 1 & 1 & 304\\ \hline
11 & 1 & 2 & 4 & 4 & 892\\
     & 2 & 3 & 3 & 3 & 1880\\
     & 3 & 4 & 2 & 2 & 1476\\
     & 4 & 5 & 1 & 1 & 120\\ \hline
12 & 1 & 1 & 5 & 5 & 28\\
     & 2 & 2 & 4 & 4 & 2242\\
     & 3 & 3 & 3 & 3 & 3312\\
     & 4 & 4 & 2 & 2 & 2242\\
     & 5 & 5 & 1 & 1 & 28\\ \hline
13 & 1 & 2 & 5 & 5 & 552\\
      & 2 & 3 & 4 & 4 & 4924\\
      & 3 & 4 & 3 & 3 & 4388\\
      & 4 & 5 & 2 & 2 & 2964\\ \hline
14 & 1 & 1 & 6 & 6 & 4\\
     & 2 & 2 & 5 & 5 & 2512\\
     & 3 & 3 & 4 & 4 & 6110\\
     & 4 & 4 & 3 & 3 & 6110\\
     & 5 & 5 & 2 & 2 & 2512\\
     & 6 & 6 & 1 & 1 & 4\\ \hline
15 & 1 & 2 & 6 & 6 & 48\\
      & 2 & 3 & 5 & 5 & 4840\\
      & 3 & 4 & 4 & 4 & 7412\\
      & 4 & 5 & 3 & 3 & 6016\\
      & 5 & 6 & 2 & 2 & 1288\\
      & 6 & 7 & 1 & 1 & 8 \\ \hline
\end{array}
\]
\caption{\label{stat} Parity populations of dots in Costas arrays of order $n\leq 15$}
\end{table}

Some of the data in the table shows a regularity which can be explained in theoretical terms. Thus,
for example, rotation of a Costas array of order $n$ clockwise through $90^{\circ}$ produces another, different
Costas array, and a dot at position $(i,j)$ is transformed to a dot at position $(j,n+1-i)$. It follows that if $n$ is even, an even/even dot is transformed to an even/odd dot, and an odd/
odd dot is transformed to an odd/even dot. This explains a symmetry in the data when $n$ is even.

\begin{lemma}\label{parpop}  Let $P$ be a permutation array of order $n$. Then we have
\[
 \#(e,o)= \#(o,e)
\]
and
\[
\begin{cases}
\#(o,o)=\#(e,e)+1, &\text{if $n$ is odd;}\\
\#(o,o)=\#(e,e), &\text{if $n$ is even.}
\end{cases}
\]
\end{lemma}

\begin{proof} Let $\Delta$ be the subset of ordered pairs $(i,j)$, where there is a dot in the position $(i,j)$
of $P$. Let $\mathcal{E}$ denote the subset of even integers lying between $1$ and $n$, and $\mathcal{O}$ the subset
of odd integers in the same interval. There are $|\mathcal{E}|$ elements $(i,j)$ in $\Delta$ with $i\in \mathcal{E}$ and, by
considering whether the second component $j$ is even or odd, we see that
\[
|\mathcal{E}|=\#(e,e)+\#(e,o).
\]
Equally, there are $|\mathcal{E}|$ elements $(k,l)$ in $\Delta$ with $l\in \mathcal{E}$ and the same argument shows that
\[
|\mathcal{E}|=\#(e,e)+\#(o,e).
\]
It follows that $\#(e,o)=\#(o,e)$. It is also clear that
\[
|\mathcal{O}|=\#(o,o)+\#(o,e).
\]
Since $|\mathcal{E}|=|\mathcal{O}|-1$ if $n$ is odd, and $|\mathcal{E}|=|\mathcal{O}|$ if $n$ is even, the equations relating $\#(e,e)$ and $\#(o,o)$ in the
two cases are immediate.
\end{proof}

We turn now to the question of enumerating the even/even dots in a Golomb-Costas array.

\begin{lemma} \label{eesq} Let $C$ be a Golomb-Costas array of order $q-2$, where $q$ is a power of an odd prime.
Then the number of even/even dots in
$C$ equals the number
of non-identity elements $z$ in $\mathbb{F}_q^*$ such that $z$ and $1-z$ are squares.
\end{lemma}

\begin{proof} We may suppose that $C$ is defined using generators $\alpha$ and $\beta$
of $\mathbb{F}_q^*$. Suppose that $C$ has a dot at an even/even position
$(i,j)$. Then we may write $i=2u$, $j=2v$ for unique positive integers
$u$ and $v$, and by definition of $C$,
\[
\alpha^{2u}+\beta^{2v}=1.
\]
It follows that if we set $z=\alpha^{2u}$, $z$ is a square and $1-z=\beta^{2v}$
is also a square. Thus an even/even dot in $C$ determines
an element $z$ with the stated property.

Conversely, suppose that $w$ is a non-identity element in $\mathbb{F}_q^*$ such that
$w$ and $1-w$ are both squares. Then we may write $w=\alpha^{2s}$ and $1-w=\beta^{2t}$
for unique positive integers $s$ and $t$, with $1<2s, 2t<q-2$. Clearly,
\[
\alpha^{2s}+\beta^{2t}=1,
\]
and thus $C$ has an even/even dot in position $(2s,2t)$.
\end{proof}

\begin{corollary}\label{eqpop} Let $q$ be a power of an odd prime.
Then all
Golomb-Costas arrays of order $q-2$ have equal numbers
of even/even dots, and likewise equal numbers of odd/odd dots, even/odd dots and odd/even dots.
\end{corollary}

We note that there are examples of Costas arrays of order $q-2$, where $q$ is a power of a prime,
which do not exhibit the same regularity with respect to even/even dots, etc. These arrays, of course, are not constructed by field-theoretic means.

To  enumerate the elements with the property described in Lemma \ref{eesq}, let $\mathbb{F}_q^\dagger$ denote the subset of $\mathbb{F}_q$ obtained by deleting 0 and 1. Let $S$ denote the subset of squares in $\mathbb{F}_q^\dagger$ and $NS$ the subset of non-squares. It is well known that $|NS|=(q-1)/2$ and thus $|S|=(q-3)/2$.

We now partition $\mathbb{F}_q^\dagger$ into four subsets $A$, $B$, $C$ and $D$ defined by
\medskip

\( {\setlength\extrarowheight{3pt}\begin{array}{lll}
A&=&\{\, z\in \mathbb{F}_q^\dagger: z\in S \mbox{ and }1-z\in S\,\}\\
B&=&\{\, z\in \mathbb{F}_q^\dagger: z\in NS \mbox{ and }1-z\in NS\,\}\\
C&=&\{\, z\in \mathbb{F}_q^\dagger: z\in S \mbox{ and }1-z\in NS\,\}\\
D&=&\{\, z\in \mathbb{F}_q^\dagger: z\in NS \mbox{ and }1-z\in S\,\}.
\end{array}} \)

\medskip
Next, we introduce two permutations $\sigma$ and $\tau$ of  $\mathbb{F}_q^\dagger$ defined by
\[
\sigma(z)=1-z,\quad \tau(z)=z^{-1}.
\]
Provided $q\ge 5$, the two permutations generate a subgroup $G$ of the group of permutations of $\mathbb{F}_q^\dagger$, isomorphic to the symmetric group $S_3$. $G$ permutes the subsets $A$, $B$, $C$ and $D$ introduced above, and this fact enables us to calculate the size of each subset.

\begin{lemma}\label{gpop} With the notation introduced above, the following hold.

\item{(a)} If $q\equiv 1\pmod {4}$, then
\[
\sigma(A)=\tau(A)=A,\quad \sigma(C)=D,\quad \tau(B)=D.
\]
Thus $|B|=|C|=|D|$ in this case.
\medskip
\item{(b)} If $q\equiv 3\pmod {4}$, then
\[
\tau(A)=C,\quad \sigma(C)=D,\quad \sigma(B)=\tau(B)=B.
\]
Thus $|A|=|C|=|D|$ in this case.
\end{lemma}

\begin{proof} We omit the details of the proof, but note that the key facts are that $-1$ is in $S$
if and only if $q\equiv 1\pmod{4}$ and the product of two elements of $NS$ is in $S$.
\end{proof}

Since
\[
|A|+|C|=|S|=\frac{q-3}{2},\quad |B|+|D|=|NS|=\frac{q-1}{2},
\]
Lemma \ref{gpop} enables us to find the size of each of the four subsets. We summarize the details below.

\[
|A|=(q-5)/4,\quad |B|=|C|=|D|=(q-1)/4
\]
when $q\equiv 1\pmod{4}$, and
\[
|A|=|C|=|D|=(q-3)/4,\quad |B|=(q+1)/4
\]
when $q\equiv 3\pmod{4}$.

Given Lemma \ref{eesq}, we readily deduce the
following information about the number of dots with a given parity in a Golomb-Costas array.

\begin{theorem}\label{gpopth} Let $C$ be a Golomb-Costas array of order $q-2$, where $q$ is a power of an odd prime.  Suppose that the number of
even/even, odd/odd, even/odd and odd/even dots in $C$ are $\#(e,e)$, $\#(o,o)$, $\#(e,o)$ and $\#(o,e)$, respectively. Then we have
\[
\#(e,e)=(q-5)/4,\quad \#(o,o)=\#(e,o)=\#(o,e)=(q-1)/4
\]
when $q\equiv 1\pmod{4}$, and
\[
\#(e,e)=\#(e,o)=\#(o,e)=(q-3)/4,\quad \#(o,o)=(q+1)/4
\]
when $q\equiv 3\pmod{4}$.
\end{theorem}

Theorem \ref{gpopth} shows that, in a Golomb-Costas array, the dots are as uniformly distributed into the four
positions as it is possible to achieve. Again, Costas arrays which are not constructed by such
algebraic methods need not show such uniformity.

It is of interest to enquire whether any such regularity can be detected in Golomb-Costas arrays of \textit{even} order $q-2$ (in other words, in those defined using a finite field of order a power of 2).
It is straightforward to show that when $q$ is a power of 2, the difference $\#(e,e)-\#(o,e)$ is an odd integer and examination of all the Golomb-Costas arrays of order $2^9-2$ and $2^{10}-2$, respectively,
reveals that this integer takes all odd values lying between 35 and $-35$ in the first case, and all odd values lying between 79 and $-79$ in the second. We have not found
any theoretical explanation of this numerical data.

Next, we consider the number of  even/even  and odd/even dots in a Welch-Costas array
of order $p-1$, where $p$ is an odd prime.
Here, the results turn out to be much more interesting when $p\equiv 3\mod{4}$. Recall that
given a generator $\alpha$ and arbitrary element $c$ in $\mathbb{F}_q^*$, we have defined
a Welch-Costas array of order $p-1$ in which
an even/even dot $(2i,2j)$ corresponds to the equation
$$
c\alpha^{2i}\equiv 2j\bmod p.
$$
This means that if $c$ is a square, the integer $2j$ is an even {\it quadratic residue\/} mod $p$, whereas if $c$ is a non-square, $2j$ is an even {\it non-quadratic residue\/} mod $p$.
Similarly, an odd/even
position $(2i+1,2j)$ corresponds to the equation
\[
c\alpha^{2i+1}\equiv 2j\pmod {p},
\]
which implies that $2j$ is an even non-quadratic residue mod $p$ if $c$ is a square, and an
even quadratic residue mod $p$ if $c$ is a non-square. It follows that, for a Welch-Costas
array, we have
\[
\#(e,e)-\#(o,e)=\pm(E-F),
\]
where $E$ is the number of even elements in $\Omega$ which are quadratic residues mod $p$ and $F$ is the number of even elements in $\Omega$ which are
non-quadratic residues modulo $p$.

\begin{lemma} \label{resnum} Suppose that $p\equiv 1\pmod{4}$. Then the number of even quadratic residues mod $p$ equals the number of even
non-quadratic residues modulo $p$.
\end{lemma}

\begin{proof} We use the fact that, under the given hypothesis, $-1$ is a square mod $p$.
Let $2j$ be an even integer in $\Omega$ which is a square modulo $p$. Then $-2j$ is also a
square mod $p$ and hence $p-2j$ is an odd quadratic residue in $\Omega$. This association shows that there are equal numbers of odd and even quadratic residues in $\Omega$. Now let
$\alpha$ be a
generator of  $\mathbb{F}_p^*$ and consider
the Welch-Costas array of order $p-1$ defined by the permutation $\sigma$ of $\Omega$
where $\sigma(i)=\alpha^i\equiv j\pmod {p}$. We know that $\#(o,e)$ equals the number of even elements in
$\Omega$ which are non-squares, and it is straightforward to see that $\#(e,o)$ equals the number
of odd elements in $\Omega$ which are squares. Since  $\#(o,e)=\#(e,o)$, it follows
from our earlier deduction that the number of even elements in $\Omega$ which are squares mod $p$
equals the number of even elements in $\Omega$ which are non-squares mod $p$, as claimed.
\end{proof}
We now deduce the following analogue of Theorem \ref{gpopth} for Welch-Costas arrays of order $p-1$ when
$p\equiv 1\pmod{4}$. Note that we have complete uniformity in the distribution of dots.

\begin{theorem} \label{wpopth1mod4}  Given a Welch-Costas array of order $p-1$, where $p$ is a prime with
$p\equiv 1\pmod{4}$,  we have
\[
\#(e,e)=\#(o,o)=\#(e,o)=\#(o,e)=(p-1)/4.
\]
\end{theorem}
The main result of this paper is that, for primes $p$ satisfying $p\equiv 3\pmod{4}$,
the situation is much more complicated than that described in Theorems \ref{gpopth} and \ref{wpopth1mod4}. We first prove a subsidiary lemma.

\begin{lemma}\label{sublem}  Let $p$ be an odd prime and let $V$, $N$ be the number of integers in the open interval $(0,p/2)$ which are squares, non-squares, respectively,
modulo $p$. Then if $p\equiv 7\pmod{8}$, $V=E$ and $N=F$, where $E$ and $F$ are described
before the statement of  Lemma \ref{resnum}. If $p\equiv 3\pmod{8}$, $V=F$ and $N=E$.
\end{lemma}

\begin{proof} Suppose first that $p\equiv 7\pmod{8}$. Then 2 is a square mod $p$. It follows
that if $x$ is an integer in the open interval $(0,p/2)$ which is a square mod $p$, $2x$ is an even
integer in $\Omega$ which is a square mod $p$. Furthermore, any even integer in $\Omega$ which is a square mod $p$ clearly arises by this doubling process. This proves that $V=E$, and the proof that
$N=F$ is identical. On the other hand, if $p\equiv 3\pmod{8}$, 2 is a non-square mod $p$, and then the doubling process maps the integers in the open interval $(0,p/2)$ which are squares mod $p$ onto
the even integers in $\Omega$ which are non-squares mod $p$. This proves that $V=F$ in this case,
and likewise, $N=E$.
\end{proof}

To state the next result, we need to recall the (complicated) concept of the \emph{class number} of an algebraic number field:

\begin{definition}
The class number of an algebraic number field equals the number of ideal classes in the underlying ring of algebraic integers. Equivalently, it is the order of the ideal class group of the number field.
\end{definition}

For example, the class number is 1 if and only if the ring of integers admits unique factorization into prime elements.

We now apply Dirichlet's class number formula (\cite{BS}, Theorem 4, p.346)  to deduce the following result.

\begin{theorem} \label{dir} Let $p$ be a prime and let
$h(-p)$ denote the class number of the imaginary
quadratic field $\mathbb{Q}(\sqrt{-p})$. In a Welch-Costas array of order
$p-1$, the following hold.

\item{(a)} If $p\equiv 7\pmod {8}$, then
\[
\#(e,e)-\#(o,e)=\pm h(-p).
\]

\item{(b)}
If $p\equiv 3\pmod {8}$, then
\[
\#(e,e)-\#(o,e)=\pm 3h(-p).
\]
\end{theorem}
The class number $h(-p)$ has been much studied and has been the subject of many deep results. By a result of Oesterl\'e, $h(-p)\ge 1/55 \log p$, \cite{IR}, p.361. It is known that the largest primes $p\equiv 3 \pmod{4}$ for which $h(-p)=1$ and 3, respectively, are 163 and 907, and that there is no prime $p\equiv 3 \pmod{4}$ with $h(-p)=2$. However, the exact value of $h(-p)$ seems to be very unpredictable in general. 

\section{Acknowledgements}

The authors would like to thank the anonymous referees for their valuable comments and corrections. They are also grateful to John Murray (NUI, Maynooth) for providing the argument for the proof of Lemma \ref{parpop}.

\end{document}